\documentclass[12pt]{article}
\usepackage{amsmath,amsthm,amsfonts,amssymb}
\usepackage{fullpage}
\usepackage{multirow}
\usepackage{array}
\usepackage{bbm}
\usepackage[noblocks]{authblk}
\usepackage{verbatim}
\usepackage{tikz}
\usepackage{float}
\usepackage{enumerate}
\usepackage{diagbox}
\usepackage{soul}
\usepackage{url}
\usepackage{mathtools}
\usepackage{ulem}
\usepackage{hyperref}
\usepackage{listings}
\usepackage{titling}
\usepackage{blindtext}

\def\modd#1 #2{#1\ \mbox{\rm (mod}\ #2\mbox{\rm )}}

\DeclareMathOperator{\LCM}{LCM}

\newtheorem{theorem}{Theorem}[section]

\newtheorem{lemma}[theorem]{Lemma}

\newtheorem{conjecture}[theorem]{Conjecture}

\theoremstyle{definition}

\newcommand{\Z}{\mathbb{Z}}

\title{Covering systems where the prime divisors of all moduli are only $2$, $3$, or $5$}

\author{Joshua Harrington, Jonah Klein, Joshua Lowrance, and Ognian Trifonov}

\begin{document}

\maketitle
\begin{abstract}
We try to find all quadruples of positive integers $(m,a,b,c)$ with $a \geq b \geq c$ such that there exists a distinct covering system with minimum modulus $m$ and least common multiple of the moduli $2^a 3^b 5^c$. We obtain complete description of all such quadruples when $m=2,3,4,5$, or $6$, except when $m=6$ and $b=c=1$. We also show that if the LCM of the moduli has  only $2$, $3$, or $5$ as prime divisors, then $m \leq 9$ and contruct a distinct covering system with $m=8$, $a=8$, $b=3$, and $c=2$. When a covering system exsists for a quadruple $(m,a,b,c)$ we provide an example. Nonexistence of covering systems is established via integer programming or by using a new estimate on the density of a set covered by a system of congruences.
\end{abstract}

\section{Introduction} \label{intro}

A  covering system ${\mathcal C}$  is a set of congruences $x \equiv \modd{r_i}  {n_i}$, $i = 1,\ldots ,k$,  such that     every integer satisfies at least one of the congruences. 

We can assume that ${\displaystyle 1 \leq n_1 \leq \cdots \leq n_k}$.  We will denote the least modulus $n_1$ of the covering by $m$, the largest modulus $n_k$ by $M$, and the least common multiple of all moduli by $L({\mathcal  C}) = L$. A distinct covering (system) is a covering where all moduli are distinct integers greater than $1$. 
The following well-known distinct covering system was exhibited by Davenport in 1952.
\begin{equation} \label{m2M12}
x \equiv \modd{1} {2}, \
x \equiv \modd{2} {4}, \ 
 x \equiv \modd{0} {3}, \ 
x \equiv  \modd{4} {6}, \
x \equiv \modd{8} {12}
\end{equation}
In this covering system  $m=2$, $M=12$, and $L=12$.

The use of covering systems in number theory was initiated by 
Erd\H{o}s in the early 1950s.  In 1971, Krukenberg \cite{Krukenberg}, a student of John Selfridge, wrote a dissertation where he did an extensive analysis of covering systems. Krukenberg investigated the question: If in a distinct covering system the minimum modulus $m$ is fixed, how small can its largest modulus $M$ be? The following table summarizes what is known about this question for $m=2,3,4,5,6$.

\begin{center}
$\begin{array}{|c|c|c|c|c|c|c|c|}
\hline
m & 2 & 3&4&5&6\\ \hline
M&12&36&60&108&168\\ \hline
\end{array}$
\end{center}

The covering system \eqref{m2M12} is an example of a covering system with $m=2$ and $M=12$. Krukenberg \cite{Krukenberg} constructed distinct covering systems with $m=3$ and $M=36$;  with $m=4$ and $M=60$; and with $m=5$ and $M=108$.
A distinct covering system with $m=6$ and $M=168$ was constructed by the fourth author and exhibited in \cite{Klein}. 

Krukenberg showed that if in a distinct covering system $m=2$, then $M \geq 12$; also, if $m=3$, then $M \geq 36$. He claimed that he had a proof that if $m=4$, then $M \geq 60$; and he conjectured that if $m=5$, then $M \geq 108$. 

About fifty years after Krukenberg's work, Dalton and the fourth author \cite{Dalton} showed that if in a distinct covering system $m=4$, then indeed $M \geq 60$. 

In July 2025, McNew and Setty \cite{McNew} introduced  the use of integer programming for covering systems (they used the Gurobi optimizer) to check whether one can construct a distinct covering with all moduli in a specific set. This allowed them to handle sets containing close to one hundred moduli. Previously, 
even for certain sets containing twenty moduli, it was not known whether one can construct a distinct covering system with all moduli in the given set. 

Extending the work of  McNew and Setty, the second author \cite{Klein} was able to show that if in a distinct covering system $m=5$, then $M \geq 108$, confirming Krukenberg's conjecture. While very little work has been done to construct distinct covering systems with fixed $m \geq 6$ and $M$ as small as possible, the second author \cite{Klein} has conjectured that if $m=6$, then $M\geq 168$.

Krukenberg also constructed distinct covering systems with $m=3,4,5,6$ where he tried to keep the LCM of the moduli $L$ as small as possible. 
The following table summarizes the smallest known values for $L$ when the minimum modulus $m$ of the covering is between $2$ and $6$. 

\begin{center}
$\begin{array}{|c|c|c|c|c|c|c|c|c|}
\hline
m & 2 & 3&4&5&6\\ \hline
L&12&120&360&1440&5040 \\ \hline
\end{array}$
\end{center}

Erd\H{o}s constructed a distinct covering with $m=3$ and $L=120$. The distinct coverings with $m=4$,  $L=360$; $m=5$, $L=1440$; and $m=6$, $L=5040$ were constructed by Krukenberg \cite{Krukenberg}.

Dalton and the fourth author \cite{Dalton} showed that the constants $120$ and $360$ in the above table cannot be replaced by smaller constants.

\begin{theorem}[Dalton and Trifonov]\label{Dalton}
If in a distinct covering system,\\
(i) $m=3$, then $L \geq 120$;
(ii) $m=4$, then $L \geq 360$.
\end{theorem}

Using Integer Programming, the second author \cite{Klein} showed that the constants $1440$ and $5040$ in the above table cannot be replaced by smaller constants.

\begin{theorem}[Klein]\label{Klein}
If in a distinct covering system,\\
(i) $m=5$, then $L \geq 1440$;
(ii) $m=6$, then $L \geq 5040$.
\end{theorem}

Krukenberg also provided a complete description of all distinct covering systems with the least common multiple of the moduli of the form $L = 2^a 3^b$ with $a$ and $b$ positive integers.

\begin{theorem}[Krukenberg]\label{Krukenberg2}
 Let ${\mathcal C}$ be a distinct covering system with the least common multiple of the moduli of the form $L = 2^a 3^b$ with $a$ and $b$ positive integers and least modulus $m$. Then

(i) $m \leq 4$;

(ii) If $m = 3$,  then $a \geq 3$, and $b \geq 2$;

(iii) If $m=3$  and $a=3$, then $b \geq 3$;

(iv) There exist distinct covering systems with $m = 3$ for each $L\in \{2^4 3^2, 2^3 3^3\}$;

(v) If $m = 4$, then $a \geq 5$ and $b \geq 3$;

(vi) There exist distinct covering systems with $m = 4$ for each $L\in\{2^7 3^3, 2^6 3^4, 2^5 3^5\}$;

(vii) There is no distinct covering system  with $m = 4$ and $L \in \{2^6 3^3, 2^5 3^4\}$.
\end{theorem}

The main goal of this paper is to find all integers $m \geq 2$ and triples of integers $a \geq b \geq c \geq 1$ for which there exist  distinct covering systems with a least modulus $m$ and the least common multiple of its moduli $L=2^a 3^b 5^c$.

The main reason for the restriction $a \geq b \geq c$ is that we believe the following conjecture holds.

\begin{conjecture}
Let $m \geq 2$, and let $a$, $b$, and $c$ be positive integers. Let $(a_1,b_1,c_1)$ be a permutation of $(a,b,c)$ such that $a_1 \geq b_1 \geq c_1$. If there exists a distinct covering system with a least modulus $m$ and the LCM of its moduli $L=2^a 3^b 5^c$, then there exists a distinct covering system with a least modulus $m$ and the LCM of its moduli $L=2^{a_1} 3^{b_1} 5^{c_1}$.
\end{conjecture}

There has been partial progress toward proving the above conjecture. For example, see
\cite{Dalton} Lemma 13.

Next, we describe our results.

\begin{theorem} \label{1.5}
Let ${\mathcal C}$ be a distinct covering system with the least common multiple of its moduli of the form $L = 2^a 3^b 5^c$ with $a$, $b$, and $c$ positive integers such that $a \geq b \geq c$, and the least modulus $m$. Then

(i) If $m =2$, then $a \geq 2$.

(ii) There exists a distinct covering system with $m=2$, $a=2$, $b=1$, $c=1$.

(iii) If $m=3$, and $a=2$, then $b \geq 2$.

(iv) There exists a distinct covering system with $m=3$ and $a=2$, $b=2$, $c=1$ not using moduli $5$ and $180$.

(v) There exists a distinct covering system with $m=3$ and $a=3$, $b=1$, $c=1$.

\end{theorem}

The next theorem deals with the case $m=4$

\begin{theorem} \label{1.6}
Let ${\mathcal C}$ be a distinct covering system with the least common multiple of its moduli of the form $L = 2^a 3^b 5^c$ with $a$, $b$, and $c$ positive integers such that $a \geq b \geq c$, and least modulus $m=4$. Then

(i) $a \geq 3$.

(ii) If $a=3$, then $b \geq 2$.

(iii) There exists a distinct covering system with $a=3$, $b=2$, $c=1$.

(iv) If $b=1$, then $a \geq 5$.

(v) There exists a distinct covering system with $a=5$, $b=1$, $c=1$ not using modulus $240$.

\end{theorem}

The next theorem deals with the case $m=5$.

\begin{theorem} \label{1.7}
Let ${\mathcal C}$ be a distinct covering system with the least common multiple of its moduli of the form $L = 2^a 3^b 5^c$ with $a$, $b$, and $c$ positive integers such that $a \geq b \geq c$, and least modulus $m=5$. Then

(i) $b \geq 2$.

(ii) If $a=3$, then $b \geq 3$ and $c \geq 2$.

(iii) There exists a distinct covering system with $a=3$, $b=3$, and $c=2$ not using moduli which are multiples of $675$.

(iv)  If $a=4$, then $c \geq 2$.

(v)  There exists a distinct covering system with $a=4$, $b=2$, and $c=2$ not using moduli $25,100,300$, and $900$.

(vi)  There exists a distinct covering system with $a=4$, $b=3$, and $c=1$ not using modulus $180$.

(vii)  There exists a distinct covering system with $a=5$, $b=2$, and $c=1$.

\end{theorem}

The next theorem deals with the case $m=6$.

\begin{theorem} \label{1.8}
Let ${\mathcal C}$ be a distinct covering system with the least common multiple of its moduli of the form $L = 2^a 3^b 5^c$ with $a$, $b$, and $c$ positive integers such that $a \geq b \geq c$, and least modulus $m=6$. Then

(i) $a \geq 4$ and $b \geq 2$.

(ii) If $a=4$, then $b \geq 3$, and $c \geq 2$.

(iii) There exists a distinct covering system with $a=4$, $b=3$, and $c=2$  not using moduli $180$ and $900$.

(iv) If $a=5$ and $b \leq 3$, then $c \geq 2$.

(v) There exists a distinct covering system with $a=5$, $b=2$, and $c=2$.

(vi) There exists a distinct covering system with $a=5$, $b=4$, and $c=1$ not using moduli which are multiples of $405$.

(vii) If $a=6$ and $c=1$, then $b \geq 3$.

(viii) There exists a distinct covering system with $a=6$, $b=3$, and $c=1$ not using moduli $270,540$, and $720$.

\end{theorem}

To show the nonexistence of certain covering systems we will need the following theorem based on the Inclusion-Exclusion Principle. To simplify the notation, when clear from the context, we will denote $\gcd (a,b)$ by $(a,b)$.

The following theorem is somewhat similar to Lemma 2.1 of Filaseta et al. \cite{Filaseta}.

\begin{theorem} \label{1.9}
  Let $M=\{m_1,m_2,\ldots,m_n\}$ be a set of distinct integers, all greater than one such that $\LCM [m_1, \ldots, m_n] = 2^a 3^b 5^c$ for some nonnegative integers $a,b,c$.  Let $\mathcal{C}$ be a system of congruences with distinct moduli, with all moduli in $M$. Then, the density of the set of integers covered by at least one of the congruences in $\mathcal{C}$ is at most \[\sum \frac{1}{m_i}-\sum_{(m_i,m_j)=1}\frac{1}{m_im_j}+\sum_{(m_i,m_j)=(m_j,m_k)=(m_k,m_i)=1}\frac{1}{m_im_jm_k}.\] In particular, if the above upper bound is less than $1$, then $\mathcal{C}$ does not cover the integers. 
\end{theorem}

Let $S$ be an infinite set of positive integers, all greater than $1$, and
${\mathcal C} = \{ n \equiv \modd{a_k} {n_k} : n_k \in S \}$. 
For each $k\in \mathbb{N}$  let $g(k)$ be the least common multiple of $n_1, n_2, \ldots, n_k$ and let $f(k)$ be the number of integers in the interval $0\leq x < g(k)$ which satisfy at least one of the congruences $x\equiv a_i\pmod{n_i}$, $i=1,\ldots, k$. We will say that $\mathcal{C}$ is a {\it strong infinite covering system} if $\lim _{k\rightarrow \infty} \frac{f(k)}{g(k)}=1$ and  every integer satisfies at least one of the congruences in ${\mathcal C}$.
Our definition is an extension of the definition of Krukenberg \cite{Krukenberg}  that $\mathcal{C}$ is an infinite covering system if $\lim _{k\rightarrow \infty} \frac{f(k)}{g(k)}=1$.
 For example, if $p_k$ is the $k$th prime number then ${\mathcal C}=\{ x \equiv \modd{0} {p_k} | k \in {\mathbb N} \}$ is an infinite covering system but not a strong infinite covering system since $1$ and $-1$ are not covered by any of the congruences. On the other hand, ${\mathcal C}_1=\{ x \equiv \modd{r_k} {p_k} | k \in {\mathbb N} \}$ where $\{r_k \}$ is the sequence $\{0,1,-1,2,-2,\ldots\}$ is a strong infinite covering system. Final example, ${\mathcal C}_2=\{ x \equiv \modd{r_k} {(k+1)!} | k \in {\mathbb N} \}$ is not an infinite covering, although every integer satisfies at least one of the congruences, we have $f_k/g_k < e-2 < 1$ for all $k \in {\mathbb N}$. 

The simplest examples of infinite covering systems occur  when $S=\{ 2^n : n \in {\mathbb N}\}$.

\begin{theorem} \label{1.10}
There exists a strong distinct infinite covering system with the least modulus $m=6$ and all moduli of the form $2^a 3^b 5^c$ where $a$, $b$, and $c$ are nonnegative integers with $b \leq 2$ and $c \leq 1$.
\end{theorem}

Next, we show that it is possible to have $m=8$ and $L=2^a3^b5^c$.

\begin{theorem} \label{1.11}
There exists a distinct covering system with the least common multiple of its moduli $L = 2^8 3^3 5^2$  and least modulus $m=8$.

\end{theorem}

\begin{theorem} \label{1.12}
Suppose ${\mathcal C}$ is a distinct covering system with all moduli of the form $2^a 3^b 5^c$ with $a$, $b$, and $c$  nonnegative integers. Then $m \leq 9$.

\end{theorem}

The rest of the paper is organized as follows: In Section 2 we pose some open problems. In Section 3 we introduce the notation that we will be using. Section 4 contains the proofs of the theorems. In Section 5 we prove the nonexistence of certain covering systems using density considerations. Section 6 is on the use of integer programming to prove the nonexistence of certain coverings.

\section{Open Problems} \label{open pr}

{\bf Problem 1.} Establish whether or not there exists a distinct covering system with minimum modulus 9 and an LCM of its moduli of the form $2^a3^b5^c$, either finite or infinite.

We expect that such a covering does not exist.

\noindent
{\bf Problem 2.} Establish whether or not there exists a positive integer $a$ such that there exists a distinct covering system with minimum modulus $6$ and an LCM of its moduli $2^a 3^2 5$.

We conjecture that such a covering does not exist.

\noindent
{\bf Problem 3.} Find all triples of positive integers $(a,b,c)$ with $a \geq b \geq c$ such that there exists a distinct covering system with minimum modulus $m=8$ and an LCM of its moduli $L=2^a 3^b 5^c$. 

Solving the above problem presents several difficulties. First, most of the covering systems involved will contain more than one hundred congruences. Next, there is no analogue of Theorem \ref{Klein} in the case $m=8$. Finally, solving Problem 3 using integer programming is currently not computationally feasible.

\noindent 
{\bf Problem 4.} Determine whether Theorem 1.9 holds without the condition
$L=2^a3^b5^c$. If not, find less restrictive conditions for the theorem to hold.

Currently, we can show that the condition $L=2^a 3^b 5^c$ can be dropped if $n=3$. The proof of the theorem also works if one replaces $2,3,5$ in the LCM condition with any three distinct primes.

\section{Notation} \label{notation}

By the Chinese Remainder Theorem, if $n_1, \ldots, n_k$ are $k$ pairwise coprime positive integers, then a congruence modulo $n_1 \cdots n_k$ is equivalent to a system of $k$ congruences with moduli $n_1, \ldots, n_k$ respectively.

For example, since $180 = 4 \cdot 5 \cdot 9$, the congruence $x \equiv \modd{11} {180}$ is equivalent to the system of congruences $x \equiv \modd{3} {4}$, $x \equiv \modd{1} {5}$, and $x \equiv \modd{2} {9}$.
A number of authors have used a notation along the lines of $[3,1,2], \ [4,5,9]$ (an ordered list of the residues followed by an ordered set of the moduli of the congruences.)

Since the prime divisors of the moduli of the congruences we will be using are only $2$, $3$, or $5$, we will use a simplified notation. We represent the congruence $x \equiv \modd{r} {2^a3^b5^c}$ as $(k|l|m)$, where to obtain $k$ we find the remainder of $r$ when divided by $2^a$ and write the remainder in base $2$ with the digits {\bf reversed}; we proceed similarly for $l$, using $3^b$ instead and base $3$; and $5^c$ for $m$ and base $5$. We use $a$ base $2$ digits for $k$, $b$ base $3$ digits for $l$, and $c$ base $5$ digits for $m$. 

For example, $[3,2,1], \ [4,9,5]$  becomes $(11|20|1)$. 

Also, if $a=0$, we set $k=*$, if $b=0$, then $l=*$, and if $c=0$, then $m=*$.

For example, $[2,1], \ [4,5]$ becomes $(01|*|1)$. 

Finally, for brevity, we skip the trailing $*$'s. For example, we write $(0)$
for the congruence $x \equiv \modd{0} {2}$  rather than $(0|*|*)$.

The notation we are using is somewhat aligned with the view of Balister et al. \cite{Balister} and authors before them of constructing covering systems as the problem of covering all lattice points in a rectangular box (in ${\mathbb R}^3$ in our case) by certain points, lines, and planes (and hyperplanes in higher dimensions). 

\section{Proofs of the theorems} \label{proofs}

\begin{proof}[Proof of Theorem \ref{1.5}]

(i) The fact that there is no distinct covering system with $m=2$ and $L=30$ follows from the work of Krukenberg on squarefree coverings \cite{Krukenberg}; see Theorem 4.3.

(ii) This follows from the existence of a covering system with $m=2$ and $L=12$, see \eqref{m2M12}.

(iii) Here we need to show that there is no distinct covering system with $m=3$ and $L=60$. This follows from Theorem \ref{Dalton} (i).

(iv) We need to construct a distinct covering system with $m=3$ and $L=180=2^23^25$ not using moduli $5$ and $180$.
Here is an example of such a covering system. 

The congruence modulo $4$ is $(11)$ and the uncovered set is $(0)$ and $(10)$. The congruences modulo $3,6, 12$ are $(*|2)$, $(0|1)$, and $(10|1)$ leaving us with uncovered set $(0|0)$ and $(10|0)$. The congruences modulo $9,18,36$ are $(*|02)$, $(1|01)$, $(10|00)$, and the uncovered set is $(0|00,01)$. Next, the congruences modulo $10$, $15$, and $30$ are $(0|*|4)$, $(*|0|3)$, and $(0|0|2)$ leaving us with uncovered set $(0|00,01|0,1)$. We use the congruences modulo $45$ and $90$ to finish the class $\modd{1} {5}$, $(*|00|1)$ and $(0|01|1)$ leaving us with uncovered set $(0|00,01|0)$. The congruence modulo $20$, $(01|*|0)$ reduces the uncovered set to $(00|00,01|0)$. We finish the covering with the congruence modulo $60$, $(00|0|0)$. 

(v) The existence of a distinct covering system with $m=3$, $L=180$ was established by Krukenberg \cite{Krukenberg}.
\end{proof}

\begin{proof}[Proof of Theorem \ref{1.6}]
(i) Here we need to show that there is no distinct covering system with $m=4$ and $L=900$. This was done via integer programming. The computation took less than a second. For more details, see Section \ref{IP}.

(ii) The fact that there is no distinct covering system with $m=4$ and $L=120$ follows from Theorem \ref{Dalton}, see (ii). 

(iii) A distinct covering system with $m=4$ and $L=360$ was constructed by Krukenberg \cite{Krukenberg}, see Example 5.4.

(iv) The fact that there is no distinct covering system with $m=4$ and $L=240$ follows from Theorem \ref{Dalton}, see (ii). 

(v) Here we need to show the existence of a distinct covering system with $m=4$ and $L=480=2^5\cdot 3 \cdot 5$ not using modulus $240$. Here is one example. The congruences modulo $4,8,16$, and $32$ are $(11), (101), (1001)$, and $(10001)$. The uncovered set at this stage is $(0)$ and $(10000)$. The congruences modulo $6,12,24,48$, and $96$ are $(0|2)$, $(01|1)$, $(001|1)$, $(1000|2)$, $(10000|1)$, leaving us with uncovered set $(0|0)$, $(000|1)$, and $(10000|0)$. The congruences modulo $5,10,20,40,80$, and $160$ are 
$(*|*|4)$, $(0|*|3)$, $(00|*|2)$, $(000|*|1)$, $(1000|*|3)$, and $(10000|*|2)$. The uncovered set at this stage is $(0|0|0)$, $(000,01|0|1)$, $(01|0|2)$, $(000|1|0)$, and $(10000|0|0,1)$.
We finish the covering with congruences modulo $15,30,60,120$, and $480$, 
$(*|0|0)$, $(0|0|1)$, $(01|0|2)$, $(000|1|0)$, and $(10000|0|1)$.

\end{proof}

\begin{proof}[Proof of Theorem \ref{1.7}]

(i) Here we need to show that for any $a \in {\mathbb N}$, there is no distinct covering system with $m=5$ and $L=2^a \cdot 3 \cdot 5$. The proof of this is given in Lemma 5.2.

(ii) Here we need to prove that there is no distinct covering system with $m=5$ and $L=1080$ or $L=1800$. The nonexistence of a distinct covering system with $m=5$ and $L=1080$ follows from Theorem \ref{Klein} (i). The nonexistence of a distinct covering system with $m=5$ and $L=1800$ was shown via integer programming. See Section 6 for details.

(iii) We need to construct a distinct covering system with $m=5$ and $L=5400=2^33^35^2$ not using moduli which are multiples of $675$. Here is one example. The congruence modulo $8$ is $(111)$ and the uncovered set is $(0,10,110)$. Next, we use the congruences  modulo $6,12,24$ to cover the class $\modd{2} {3}$, namely $(0|2)$, $(10|2)$, and $(110|2)$. The uncovered set is $(0,10,110|0,1)$. Next, we attack $\modd{1} {3}$. The congruences modulo $9,18,36,72$ are $(*|12)$, $(0|11)$, $(10|11)$, and $(110|11)$. The uncovered set now is $(0,10,110|0,10)$. The congruences modulo $27,54,108,216$ are $(*|102)$, $(0|101)$, $(10|101)$, and $(110|101)$. The uncovered set now is $(0,10,110|0,100)$. Next, we cover two classes mod $5$ using the congruences modulo $5,10,20,40$, namely $(*|*|4)$, $(0|*|3)$, $(10|*|3)$, and $(110|*|3)$. The uncovered set  is $(0,10,110|0,100|0,1,2)$. We cover two more classes modulo $5$ with congruences modulo $15,30,60,120$, $(*|0|2)$, $(0|0|1)$, $(10|0|1)$, $(110|0|1)$ and modulo $135,270,540,1080$,  $(*|100|2)$, $(0|100|1)$, $(10|100|1)$, $(110|100|1)$. The uncovered set now is $(0,10,110|0,100|0)$. We cover the class $\modd{1} {27}$ by using the congruences modulo $90,180,360$, $(0|10|0)$, $(10|10|0)$, and $(110|10|0)$. The congruence modulo $45$ is $(*|02|0)$ leaving us with the uncovered set $(0,10,110|00,01|0)$. We cover two classes modulo $25$ using the congruences $(*|*|04)$, $(0|*|03)$, $(10|*|03)$, and $(110|*|03)$. We cover two more classes modulo $25$ with the congruences $(*|0|02)$, $(0|0|01)$, $(10|0|01)$, and $(110|0|01)$. We finish the covering with the congruences modulo $225$, $450$, $900$, $1800$, $(*|01|00)$, $(0|00|00)$, $(10|00|00)$, and $(110|00|00)$.

(iv) Here we need to show that there is no distinct covering system with $m=5$ and $L =720$. This follows from Theorem \ref{Klein} (i).

(v) We need to construct a distinct covering system with $m=5$ and $L=3600=2^43^25^2$ not using moduli $25$, $100$, $300$, and $900$. Here is one example. The congruences modulo $8$ and $16$ are $(110)$ and $(1111)$ and the uncovered set is $(0,10,1110)$. The congruences modulo $6,12,24,48$ are $(0|2)$, $(10|2)$, $(101|1)$, and $(1110|2)$. The uncovered set is $(0,10,1110|0)$ and $(0,100,1110|1)$.
 Next, we concentrate on the class $\modd{0} {3}$. The congruences modulo $9,18,36,72,144$ are $(*|02)$, $(0|01)$, $(10|01)$, $(101|00)$, and $(1110|01)$. The uncovered set is $(0,100,1110|00,1)$. We cover two residue classes modulo $5$ by using the congruences $(*|*|4)$, $(0|*|3)$, 
 $(100|*|3)$, and $(1110|*|3)$. The congruence modulo $20$ is $(01|*|2)$. The uncovered set is $(0,100,1110|00,1|0,1)$ and $(00,100,1110|00,1|2)$.
We cover two more classes modulo $5$ with the congruences $(*|1|0)$, $(*|00|0)$, $(0|1|1)$, $(100|1|1)$, $(1110|1|1)$, $(0|00|1)$, $(100|00|1)$, and $(1110|00|1)$. The uncovered set now is $(00,100,1110|00,1|2)$. The congruences $(00|1|2)$ and $(00|00|2)$ reduce the uncovered set to 
$(100,1110|00,1|2)$. We finish the covering with the congruences $(1|*|24)$, $(100|*|23)$, $(1110|*|23)$, $(*|1|22)$, $(1|1|21)$, $(100|1|20)$, $(1110|1|20)$, $(*|00|22)$, $(1|00|21)$, $(100|00|20)$, and 
$(1110|00|20)$. 

(vi) We need to construct a distinct covering system with $m=5$ and $L=2160=2^4 3^3 5$ not using modulus $180$. Here is one example. The first $11$ congruences are the same as in the covering system in (v). Recall that the uncovered set at this stage is $(0,100,1110|00,1)$. The congruences modulo $27,54,108,216,432$ are
$(*|002)$, $(0|001)$, $(100|001)$, $(1110|001)$, and $(01|000)$. The uncovered set is $(00,100,1110|000)$ and $(0,100,1110|1)$. The congruences modulo $5,10,20,40,80$ are $(*|*|4)$, $(0|*|3)$, $(00|*|2)$, $(100|*|3)$, and $(1110|*|3)$. The uncovered set is $(00,100,1110|000|0,1)$, 
$(0,100,1110|1|0,1)$,  $(01,100,1110|1|2)$, and $(100,1110|000|2)$. 
The congruences modulo $15,30,60,120,240$ are $(*|1|0)$, $(0|1|1)$, $(01|1|2)$, $(100|1|2)$, and $(1110|1|2)$. The uncovered set is 
$(00,100,1110|000|0,1)$, $(100,1110|1|1)$, and $(100,1110|000|2)$. The congruences modulo $135,270,540,1080,2160$ are $(*|000|0)$, $(1|000|2)$, $(00|000|1)$, $(100|000|1)$, and $(1110|000|1)$. The uncovered set is 
$(100,1110|1|1)$. We finish the covering with the congruences modulo $45$, $90$, $360$, $720$,  $(*|12|1)$, $(1|11|1)$, $(100|10|1)$, and $(1110|10|1)$. 

(vii) The covering with $a=5$, $b=2$, $c=1$ was constructed by Krukenberg \cite{Krukenberg}, Chapter~V.

\end{proof}

\begin{proof}[Proof of Theorem \ref{1.8}]
(i) We need to show that there is no distinct covering with $m=6$ and $a=b=c=3$. This was done via integer programming and the computation took more than $9000$ seconds. See Section  6 for details. So, $a \geq 4$. Also, Theorem 1.7 (i) implies $b \geq 2$. 

(ii) To show that if $a=4$, then $b \geq 3$, and $c \geq 2$, we need to show that there is no covering with $a=4, b=2, c=2$ or with $a=4, b=4, c=1$. The nonexistence of a distinct covering system with $m=6$, $L=2^43^25^2=3600$ follows from Theorem \ref{Klein} (ii).  The nonexistence  of a distinct covering system with $m=6$, $L=2^43^45=6480$ was established via integer programming, see Section 6.

(iii) Here we need to construct a covering system with $m=6$ and $L=10800=2^43^35^2$  not using moduli $180$ and $900$. Here is one example. All the congruences with moduli dividing $2^43^35$ are the same as in the covering in Theorem 1.7 (vi), except that we cannot use modulus $5$ (the congruence $(*|*|4)$ is missing.) The uncovered set at this stage is $(00,100,1110|000|4)$ and $(0,100,1110|1|4)$. We finish the covering using congruences moduli $d$, where $5^2 \mid d \mid 2^4 3^3 5^2$, and structure similar to the one of the covering in  Theorem 1.7 (vi). The congruences modulo $25,50,100,200,400$ are $(*|*|44)$, $(0|*|43)$, $(00|*|42)$, $(100|*|43)$, and $(1110|*|43)$. The congruences modulo $75,150,300,600,1200$ are $(*|1|40)$, $(0|1|41)$, $(01|1|42)$, $(100|1|42)$, and $(1110|1|42)$. The congruences modulo $675,1350,2700,5400,10800$ are $(*|000|40)$, $(1|000|41)$, $(00|000|41)$, 
$(100|000|42)$, and $(1110|000|42)$. The uncovered set is $(100,1110|1|41)$. We finish the covering with the congruences modulo $225$, $450$, $1800$, $3600$, $(*|12|41)$, $(1|11|41)$, $(100|10|41)$, and $(1110|10|41)$.

(iv) The nonexistence of a distinct covering system with $m=6$ and $L=4320=2^53^35$ follows from Theorem \ref{Klein} (ii).

(v) A distinct covering system with $m=6$ and $L=7200=2^53^25^2$ was constructed by Krukenberg \cite{Krukenberg} in Chapter V.

(vi) We need to construct a distinct covering system with $m=6$ and $L=12960=2^53^45$ not using moduli which are multiples of $405$. Here is one example. The congruences modulo $8,16,32$ are $(110)$, $(1110)$, and $(11111)$. The uncovered set is $(0,10,11110)$.
The congruences modulo $6,12,24,48,96$ are $(0|2)$, $(10|2)$, $(101|1)$, $(1001|1)$, and $(11110|2)$. The uncovered set is $(0,10,11110|0)$ and
$(0,1000,11110|1)$. Next, we concentrate on the class $\modd{0} {3}$. 
The congruences modulo $9,18,36,72,144,288$ are $(*|02)$, $(0|01)$, $(10|01)$, $(101|00)$, $(1001|00)$, and $(11110|01)$. The uncovered set is 
$(0,1000,11110|00,1)$. The congruences modulo $27,54,108,216,432,864$ are
$(*|002)$, $(0|001)$, $(01|000)$, $(001|000)$, $(1000|001)$, and $(11110|001)$. The uncovered set is $(000,1000,11110|000)$ and $(0,1000,11110|1)$. The congruences modulo $81,162,324,648,1296,2592$ are $(*|0002)$, $(1|0001)$, $(00|0001)$, $(000|0000)$, $(1000|0000)$, $(11110|0000)$, finishing the class $\modd{0} {3}$. The uncovered set is 
$(0,1000,11110|1)$. The congruences modulo $10,20,40,80,160$ are $(0|*|4)$, $(01|*|3)$, $(100|*|1)$, $(1000|*|4)$, and $(11110|*|4)$. The congruences modulo $15,30,60,120,240,480$ are $(*|1|0)$, $(0|1|1)$, $(00|1|3)$, $(111|1|1)$, $(1000|1|3)$, and $(11110|1|3)$. The uncovered set is $(0,1000,11110|1|2)$.
The congruences modulo $45,90,180,360,720,1440$  are $(*|12|2)$, $(0|11|2)$, $(01|10|2)$, $(001|10|2)$, $(1000|11|2)$, and $(11110|11|2)$.
The uncovered set is $(000,1000,11110|10|2)$. We finish the covering with the congruences modulo $135$, $270$, $540$, $1080$, $2160$, $4320$ which are $(*|102|2)$, $(1|101|2)$, $(00|101|2)$, $(000|100|2)$, $(1000|000|2)$, and $(11110|000|2)$.

(vii) We need to show the nonexistence of a distinct covering system with 
$m=6$ and $L=2^63^25 = 2880$. This follows from Theorem \ref{Klein} (ii).

(viii) We need to construct a distinct covering system with $m=6$ and
$L = 8640=2^63^35$ not using moduli $270,540$, and $720$. Here is one example. The congruences modulo $8,16,32,64$ are $(110)$, $(1110)$, $(11110)$, and $(111111)$. The uncovered set is $(0,10,111110)$. The congruences modulo $6,12,24,48,96,192$ are $(0|2)$, $(10|2)$, $(101|1)$, $(1001|1)$, $(10001|1)$, and $(111110|2)$. The uncovered set is $(0,10,111110|0)$ and $(0,10000,111110|1)$. The congruences modulo $9,18,36,72,144,288,576$ are $(*|02)$, $(0|01)$, $(10|01)$, $(101|00)$,$(1001|00)$, $(10001|00)$, and $(111110|01)$.
The uncovered set is $(0,10000,111110|00,1)$. The congruences modulo $27,54,108,216,432,864,1728$ are $(*|002)$, $(0|001)$, $(01|000)$, $(001|000)$, $(1000|000)$, $(10000|001)$, and $(111110|001)$. The uncovered set is $(000,111110|000)$ and $(0,10000,111110|1)$. The congruences modulo $10,20,40,80,160,320$ are $(0|*|4)$, $(00|*|3)$, $(000|*|2)$, $(1000|*|3)$, $(10000|*|4)$, and $(11110|*|4)$. The uncovered set is $(0|1|0,1)$, $(01|1|3)$, $(01,001|1|2)$, $(000|000|0,1)$, $(10000|1|0,1,2)$, and $(111110|1,000|0,1,2,3)$. The congruences modulo $15$, $30$, $60$, $120$, $240$, $480$, $960$ are $(*|1|0)$, $(0|1|1)$, $(01|1|3)$, $(100|1|1)$, $(1000|1|2)$, $(11111|1|3)$, and $(111110|1|1)$. 
The uncovered set is $(01,001|1|2)$, $(111110|1|2)$, $(000|000|0,1)$, 
and $(111110|000|0,1,2,3)$. We use the congruences modulo $45,90,180,360,1440,2880$ to finish the class $\modd{1} {3}$. These congruences are $(*|12|2)$, $(0|11|2)$, $(01|10|2)$, $(001|10|2)$, $(11111|11|2)$, and $(111110|10|2)$. The uncovered set is $(000|000|0,1)$ and $(111110|000|0,1,2,3)$. We finish the covering with the congruences modulo $135$, $1080$, $2160$, $4320$, $8640$, 
$(*|000|0)$, $(000|000|1)$, $(1111|000|1)$, $(11111|000|2)$, and $(111110|000|3)$.

\end{proof}

\begin{proof} [Proof of Theorem \ref{1.9}]
We prove the theorem by induction on $n$. 
Suppose that the congruence modulo $m_i$ is $x \equiv \modd{a_i} {m_i}$. Let $L =\LCM [m_1, \dots, m_n]$. Let $C_i$ be the set of integers in $[1,L]$ covered by the congruence modulo $m_i$, that is $$C_i = \{ t \in \mathbb{Z} : 1 \leq t \leq L, t \equiv \modd{a_i} {m_i} \}.$$
The theorem is equivalent to
\begin{equation} \label{inc-exc} |C_1 \cup \cdots \cup C_n| \leq \sum \frac{L}{m_i}-\sum_{(m_i,m_j)=1}\frac{L}{m_im_j}+\sum_{(m_i,m_j)=(m_j,m_k)=(m_k,m_i)=1}\frac{L}{m_im_jm_k}.
\end{equation}

I. Base case, $n=2$. 

By the Principle of Inclusion-Exclusion we have $|C_1 \cup C_2| = |C_1| + |C_2| - |C_1 \cap C_2|.$
Now, $|C_1| = \frac{L}{m_1}$, $|C_2| = \frac{L}{m_2}$, and if $\gcd(m_1,m_2)>1$, \eqref{inc-exc} holds. If $\gcd(m_1,m_2)=1$, then 
$|C_1 \cap C_2| = \frac{L}{m_1m_2}$ and \eqref{inc-exc} also holds. 

II. Inductive step. Let $u$ be an integer $\geq 2$. Assume that equation \eqref{inc-exc} holds for all positive integers $\leq u$. We show that it holds for $u+1$ as well. Without loss of generality assume that $\gcd(m_{u+1},m_i)=1$ for $i=1,\ldots ,l$ and
$\gcd(m_{u+1},m_i)>1$ for $i=l+1,\ldots ,u$ (otherwise we relabel the moduli).

Denote $ A = \bigcup \limits _{i=1}^{u} C_i$, $A_0 =  \bigcup \limits _{i=1}^{l} C_i$, and let $A_1=A\setminus A_0$. Then, $\bigcup \limits _{i=1}^{u+1} C_i = A \bigcup C_{u+1}$. 
We have 
\begin{equation} \label{u+1}
|A \cup C_{u+1}| =|A| + |C_{u+1}| - |A \cap C_{u+1}| \leq |A|  + \frac{L}{m_{u+1}} - |A_0 \cap C_{u+1}| =|A| + \frac{L}{m_{u+1}} - \frac{|A_0|}{m_{u+1}},
\end{equation}
since $m_{u+1}$ is relatively prime to each of the moduli $m_1, \ldots, m_l$. 
Applying the induction hypothesis to $m_1, \ldots , m_u$ we obtain
\eqref{inc-exc} with $n=u$, which is an upper bound on $|A|$.

If the lower bound 
\begin{equation} \label{A0lb}
|A_0| \geq \sum_{i=1}^l \frac{1}{m_i} - \sum_{(m_i,m_j)=1, 1 \leq i < j \leq l} \frac{1}{m_i m_j},
\end{equation}
holds, then substituting  in \eqref{A0lb} the lower bound on $|A_0|$ and the upper bound on $|A|$,  we obtain \eqref{inc-exc} with $n=u+1$ completing the proof. 

Next, we show that \eqref{A0lb} holds when $L=2^a3^b5^c$ (with minor adjustments to the original congruences, if needed, and proper choice of $m_{u+1}$).

First, we can assume that no two congruences with moduli that are powers of $2$ intersect, for example, we do not have the congruences $x \equiv \modd{0} {4}$ and $x \equiv \modd{0} {8}$ in $\mathcal {C}$. Note that all integers covered by the congruence $x \equiv \modd{0} {8}$ are covered by the congruence $x \equiv \modd{0} {4}$. So, if we replace the congruence $x \equiv \modd{0} {8}$ with the congruence the congruence $x \equiv \modd{a} {8}$ where $a=1,2,3,5,6$, or $7$, the congruences modulo $4$ and modulo $8$ no longer intersect. Moreover, the size of the covered set does not decrease after modifying the congruence modulo $8$. So, going through the congruences  with moduli that are powers of $2$ in increasing order and modifying congruences when needed we can ensure that the congruences with moduli which are powers of $2$ do not intersect and that we have not decreased the size of the covered set in the process. 
Similarly, we can ensure that the congruences with moduli that are powers of $3$ do not intersect, and the congruences with moduli that are powers of $5$ do not intersect, as well,

Next, we consider several cases.

Case 1. At least one of the moduli $m_1, \ldots, m_{u+1}$ is divisible by $30$. Since the estimate \eqref{inc-exc} is a symmetric function of 
$m_1, \ldots, m_{u+1}$, we can assume that $m_{u+1}$ is one of the moduli divisible by $30$. Then, $A_0$ is the empty set and \eqref{A0lb} holds.

Case 2. None of the moduli is divisible by $30$ but at least one is divisible by $6$, $10$, or $15$. 

First, consider the case when $m_{u+1}$ is divisible by $6$ but not by $30$.
If $A_0$ is the empty set, we are done. Otherwise, the congruences in $A_0$ have moduli which are powers of $5$ and ${\displaystyle |A_0| = \sum_{i=1}^l \frac{1}{m_i}}$, and \eqref{A0lb} holds. 

We proceed in a similar way when there is a modulus divisible by $10$ or $15$ but not by $30$.

Case 3. None of the moduli is divisible by $6$, $10$, or $15$. In this case, each modulus is a power of $2$, a power of $3$, or a power of $5$. 

First, consider the case where one of the moduli is a power of $2$. Pick $m_{u+1}$ to be that modulus. Then, the congruences in $A_0$ have moduli that are a power of $3$ or a power of $5$. In this case,  we have 
$$|A_0| = \sum_{i=1}^l \frac{1}{m_i} - \sum_{(m_i,m_j)=1, 1 \leq i < j \leq l} \frac{1}{m_i m_j}, $$
and \eqref{A0lb} holds. 

We proceed similarly if one of the moduli is a power of $3$ or a power of $5$. 
\end{proof}

\begin{proof}[Proof of Theorem \ref{1.10}]
First, we introduce a new notation. Let $a$ be a positive integer, let $d$ be an odd positive integer, and let $r$ be an integer. We denote by $2^a \uparrow  d$ the infinite set of congruences $\{x \equiv \modd{r_j} {2^{a+j-1}d} \ | \ j \in{\mathbb N}\}$, where $r_j \equiv \modd{r} {d}$ and $r_j \equiv \modd{r+2^{a+j-2}} {2^{a+j-1}}$. It is clear that no integer is covered by more than one of the above congruences since if $j < k$ the $j$th and the $k$th congruences are in distinct classes modulo $2^{a+j-1}$. Therefore, the congruences in $2^a \uparrow d$ cover the residue class 
$\modd{r} {2^{a-1}d}$ (the density of integers covered by the congruences is $1/(2^{a-1}d)$, the same as the density of the integers in the residue class $\modd{r} {2^{a-1}d}$ and all congruences are in this residue class.)

For example if $a=1$, $d=1$, and $r=0$, $2 \uparrow$ is the set of congruences $(1)$, $(01)$, $(001), \ldots$ and the congruences cover all integers. 

Another example. If $a=3$, $d=15$, and $r=38$, the set $2^3 \uparrow 15$ is 
$(011|2|3)$, $(0101|2|3)$, 
$ (01001|2|3), \ldots $ and the congruences cover the residue class $(01|2|3)$, that is $\modd{38} {60}$.

Here is the construction of the infinite covering system we need. With the congruences in $2^3 \uparrow$ we cover the residue class $(11)$. The uncovered set is $(0,10)$. The congruences modulo $6$ and $12$ are $(0|2)$ and $(10|2)$. With the congruences in $2^3 \uparrow 3$ we cover $(10|1)$.
The uncovered set is $(0|0,1)$ and $(10|0)$. Next we concentrate on the residue class $\modd{0} {3}$. The congruences modulo $9,18,36$ are 
$(*|02)$, $(0|01)$, and $(10|01)$. With the congruences in $2^3 \uparrow 3^2$ we cover $(10|00)$. The uncovered set now is $(0|00,1)$. The congruences modulo $10$ and $20$ are $(0|4)$, $(01|3)$ and we cover $(00|3)$ with the congruences $2^3 \uparrow 5$. The uncovered set is
$(0|00,1|0,1,2)$. The congruences modulo $15$ and $30$ are $(*|1|2)$ and 
$(0|1|1)$. We cover $(0|1|0)$ with the congruences in $2^2 \uparrow 15$. 
The congruences modulo $45$ and $90$ are $(*|00|2)$ and $(0|00|1)$. 
We cover $(0|00|0)$ with the congruences in $2^2 \uparrow 45$, finishing the covering.

It is not difficult to see that if we order the above congruences in increasing order of their moduli, the resulting system of congruences is a strong distinct infinite covering. 

\end{proof}

Note that the above covering is tight, the construction above used every available modulus. We conjecture that for every positive integer $a$ there is no distinct covering with $m=6$ and $L=2^a 3^2 5$.

\begin{proof}[Proof of Theorem \ref{1.11}]
Here we construct a distinct covering system with $m=8$ and $L=2^83^35^2=172800$. This covering is much more involved than the rest of the coverings in this paper. Some reasons are the large number of congruences used, the nontrivial structure of the covering,   and the fact that the covering is relatively tight. Although we do not use all available moduli, we expect that there is no distinct covering system with $m=8$ and $L$ a proper divisor of $2^83^25$. The construction of the covering is in three stages. In Stage 1, we use the available moduli of the form $2^a3^b$ to cover the odd integers and most of the residue class $\modd{0} {3}$. We save the congruence modulo $2^8$ for the second stage since it makes the structure of stages 2 and 3 simpler. In stage 2, we use congruences with moduli divisible by $5$ but not by $25$ to cover $4$ residue classes modulo $5$. In stage 3 we use the moduli left (all divisible by $25$) to cover the remaining residue class modulo $5$.

The congruences modulo $8,16,32,64,128$ are
$(110)$, $(1110)$, $(11110)$, $(111110)$, $(1111110)$ and the uncovered set is $(0,10,1111111)$. The congruences modulo $12,24,48,96,192,384,768$ are $(10|2)$, $(101|1)$, $(1001|1)$, $(10001|1)$, $(111111|2)$, $(1111111|1)$, and $(11111110|0)$. The uncovered set is $(0|0,1,2)$, $(10|0)$, $(10000|1)$, and $(11111111|0)$. The congruences modulo $9,18,36,72,144,288,576,1152, 2304$ are $(*|02)$, $(0|01)$, $(10|01)$, $(101|00)$, $(1000|12)$, $(10000|11)$, $(100001|10)$, $(1111111|01)$, and $(11111111|00)$, The uncovered set is 
$(0|00,1,2)$, $(100|00)$, and $(100000|10)$. The congruences modulo $27,54,108,216,432,864,1728$ are $(*|002)$, $(0|001)$, $(10|001)$, $(100|000)$, $(1000|102)$, $(10000|101)$, and $(100000|100)$. The uncovered set is $(0|000,1,2)$. Now, we use the congruence modulo $256$, 
$(01111111)$. The uncovered set is\\ $(00,010,0110,01110,011110,0111110,01111110|000,0,2)$.
For brevity, we denote the expression $00,010,0110,01110,011110,0111110,01111110$ by $s$. So, the uncovered set is $(s|000,1,2)$. This completes the first stage of the covering. 
The congruences modulo $5\cdot 2, 5\cdot 2^2, \ldots, 5\cdot 2^8$ are
$(0|*|4)$, $(00|*|3), (010|*|3), \ldots, (01111110|*|3)$. The uncovered set is $(s|000,1,2|0,1,2)$. The congruences modulo $15$ and $30$ are $(*|2|2)$ and
$(0|1|2)$. The congruences modulo $15\cdot 2^2, \ldots, 15 \cdot 2^8$ are
$(00|2|1), (010|2|1),\ldots, (01111110|2|1)$. The uncovered set is $(s|000|0,1,2)$, $(s|1|0,1)$, and $(s|2|0)$. The congruences modulo 
$45,90$ are $(*|12|1)$ and $(0|11|1)$. The congruences modulo 
$45 \cdot 2^2, \ldots, 45 \cdot 2^8$ are $(00|10|1), \ldots, (01111110|10|1)$. The uncovered set is $(s|000|0,1,2)$ and $(s|1,2|0)$. 
The congruences modulo $135,270$ are $(*|000|2)$ and $(0|000|1)$. 
The uncovered set is $(s|000,1,2|0)$. 

It may appear that we are in the same position we were at the beginning of stage 2. However, here we can use a congruence modulo $5^2$ while we could not use congruence modulo $5$ in stage 2 since $m=8$. 

The congruence modulo $25$ is $ (*|*|00)$. After that the structure of the covering is the same as in stage 2. The congruences modulo $5^2\cdot 2, 5^2\cdot 2^2, \ldots, 5^2\cdot 2^8$ are
$(0|*|04)$, $(00|*|03), (010|*|03), \ldots, (01111110|*|03)$. The uncovered set is $(s|000,1,2|01,02)$.
The congruences modulo $75$ and $150$ are $(*|2|02)$ and
$(0|1|02)$. The congruences modulo $75\cdot 2^2, \ldots, 75 \cdot 2^8$ are
$(00|2|01), (010|2|01),\ldots, (01111110|2|01)$. The uncovered set is $(s|000|01,02)$ and $(s|1|01)$. The congruences modulo 
$225,450$ are $(*|12|01)$ and $(0|11|01)$. The congruences modulo 
$225 \cdot 2^2, \ldots, 225 \cdot 2^8$ are $(00|10|01), \ldots, (01111110|10|01)$. The uncovered set is $(s|000|1,2)$.
We finish the covering with the congruences modulo $675,1350$, $(*|000|02)$ and $(0|000|01)$.

\end{proof}

\section{Nonexistence of certain covering systems} \label{nonexist}

In this section, we will be using Theorem \ref{1.9} to show the nonexistence of certain covering systems.

In particular, with an LCM of $2^a3^b5^c$ and moduli all divisors of $L$ greater than $1$, we see that \[\sum \frac{1}{m_i}=\left(\sum_{n=0}^a\frac{1}{2^n}\right)\left(\sum_{n=0}^b\frac{1}{3^n}\right)\left(\sum_{n=0}^c\frac{1}{5^n}\right)-1,\]    \[\sum_{\gcd(m_i,m_j)=1}\frac{1}{m_im_j}=\left(\sum_{n=1}^a\frac{1}{2^n}\right)\left(\sum_{n=1}^b\frac{1}{3^n}\right)+\left(\sum_{n=1}^a\frac{1}{2^n}\right)\left(\sum_{n=1}^c\frac{1}{5^n}\right)+\left(\sum_{n=1}^b\frac{1}{3^n}\right)\left(\sum_{n=1}^c\frac{1}{5^n}\right)+\]    \[3\left(\sum_{n=1}^a\frac{1}{2^n}\right)\left(\sum_{n=1}^b\frac{1}{3^n}\right)\left(\sum_{n=1}^c\frac{1}{5^n}\right), \text{ and}\] \[\sum_{(m_i,m_j)=(m_j,m_k)=(m_k,m_i)=1}\frac{1}{m_im_jm_k}=\left(\sum_{n=1}^a\frac{1}{2^n}\right)\left(\sum_{n=1}^b\frac{1}{3^n}\right)\left(\sum_{n=1}^c\frac{1}{5^n}\right).\] Note that given a minimum modulus, certain terms will be removed from the summations above, but those will be dealt with in each specific case.

\begin{lemma}
    There is no distinct covering system of the integers with LCM $2^a\cdot3\cdot5$ for all $a \in \mathbb{N}$ with a minimum modulus of $5$.
\end{lemma}

\begin{proof}
    First, consider $\sum \frac{1}{m_i}$: since we no longer have the moduli $2,3,4$, we can simply subtract them from the sum: \[\sum \frac{1}{m_i}=\left(1+\frac{1}{2}+\cdots + \frac{1}{2^a}\right)\left(1+\frac{1}{3}\right)\left(1+\frac{1}{5}\right)-1-\frac{1}{2}-\frac{1}{3}-\frac{1}{4}=\frac{67}{60}- \frac{8}{5 \cdot 2^a}.\]

Note that if $a \leq 3$, $\sum \frac{1}{m_i}<1$ and the lemma follows. So, we can assume $a \geq 4$.

    Second, consider ${\displaystyle \sum_{\gcd(m_i,m_j)=1}\frac{1}{m_im_j}}$. If $\gcd(m_i , m_j)=1$, then at least one of the moduli $m_i$, $m_j$
    is  odd. The only odd divisors of $2^a \cdot 3 \cdot 5$ which are $\geq 5$ are $5$ and $15$. So, if $\gcd(m_i,m_j)=1$, then either one of $m_i, m_j$ is $5$ and the other is one of $2^3, \ldots,2^a, 3\cdot 2, \ldots, 3\cdot 2^a$, or one is $15$ and the other is one of $2^3, \ldots, 2^a$. Thus,
    $$\sum_{\gcd(m_i,m_j)=1}\frac{1}{m_im_j}= \frac{1}{5}\left ( \frac{1}{2^3} + \cdots + \frac{1}{2^a} + \frac{1}{3 \cdot 2} + \cdots + \frac{1}{3 \cdot 2^a}\right )+\frac{1}{15}\left (  \frac{1}{2^3} + \cdots + \frac{1}{2^a} \right )=$$
    $$= \frac{1}{5} \left ( \frac{1}{4} - \frac{1}{2^a} + \frac{1}{3} - \frac{1}{3\cdot 2^a} \right ) + \frac{1}{15} \left ( \frac{1}{4} - \frac{1}{2^a} \right ) = \frac{8}{60} - \frac{1}{3\cdot 2^a}.$$

    Finally, consider ${\displaystyle \sum_{\gcd(m_i,m_j)=\gcd(m_j,m_k)=\gcd(m_k,m_i)=1}\frac{1}{m_im_jm_k}}$. Since $m_i$, $m_j$,and $m_k$ have to be pairwise coprime, at least two of the moduli need to be odd integers, and the only odd moduli we can use are $5$ and $15$. However, $\gcd(5,15)>1$, so the above sum is empty. 

    Now, Theorem \ref{1.9} gives us an  upper bound for the density of the covered set: $$\frac{67}{60}-\frac{8}{5 \cdot 2^a} - \frac{8}{60}+ \frac{1}{3\cdot 2^a} = \frac{59}{60} - \frac{19}{15 \cdot 2^a} < 1.$$ 
    Therefore, there is no covering of the integers with $m=5$ and $L=2^a \cdot 3 \cdot 5$.
\end{proof}

Note that this proof also shows there is no distinct infinite covering system with all moduli $\geq 5$ and of the form $2^n\cdot d$ where $d \mid 15$.

\begin{proof} [Proof of Theorem \ref{1.12}]
    We begin by recalling $\sum_{n=0}^\infty \frac{1}{p^n}=\frac{1}{1-\frac{1}{p}}$ for $p>1$. First, consider $\sum \frac{1}{m_i}$: since we no longer have the moduli $2,3,4,5,6,8$, and $9$, we can simply subtract them from the sum: \[\sum \frac{1}{m_i}=\left(2\right)\left(\frac{3}{2}\right)\left(\frac{5}{4}\right)-\sum_{n=1}^6 \frac{1}{n}-\sum_{n=8}^9 \frac{1}{n}=\frac{383}{360}.\]

    Second, consider ${\displaystyle \sum_{\gcd(m_i,m_j)=1}\frac{1}{m_im_j}}$. All the moduli we are losing are powers of primes except for $6$, so we can say: \[\sum_{\gcd(m_i,m_j)=1}\frac{1}{m_im_j}=\left(1-\frac{1}{2}-\frac{1}{4}-\frac{1}{8}\right)\left(\frac{1}{2}-\frac{1}{3}-\frac{1}{9}\right)+\left(1-\frac{1}{2}-\frac{1}{4}-\frac{1}{8}\right)\left(\frac{1}{4}-\frac{1}{5}\right)\] \[+\left(\frac{1}{2}-\frac{1}{3}-\frac{1}{9}\right)\left(\frac{1}{4}-\frac{1}{5}\right)+(1)\left(\frac{1}{2}\right)\left(\frac{1}{4}-\frac{1}{5}\right)-\frac{1}{6}\left(\frac{1}{4}-\frac{1}{5}\right)+\] \[(1)\left(\frac{1}{2}-\frac{1}{3}-\frac{1}{9}\right)\left(\frac{1}{4}\right)+\left(1-\frac{1}{2}-\frac{1}{4}-\frac{1}{8}\right)\left(\frac{1}{2}\right)\left(\frac{1}{4}\right)=\frac{179}{2880}\]

    Finally, consider ${\displaystyle \sum_{(m_i,m_j)=(m_j,m_k)=(m_k,m_i)=1}\frac{1}{m_im_jm_k}}$, which is simply \[\sum_{(m_i,m_j)=(m_j,m_k)=(m_k,m_i)=1}\frac{1}{m_im_jm_k}=\left(1-\frac{1}{2}-\frac{1}{4}-\frac{1}{8}\right)\left(\frac{1}{2}-\frac{1}{3}-\frac{1}{9}\right)\left(\frac{1}{4}-\frac{1}{5}\right)=\frac{1}{2880}\]

    Thus we see: \[\sum \frac{1}{m_i}-\sum_{(m_i,m_j)=1}\frac{1}{m_im_j}+\sum_{(m_i,m_j)=(m_j,m_k)=(m_k,m_i)=1}\frac{1}{m_im_jm_k}=\frac{383}{360}-\frac{179}{2880}+\frac{1}{2880}=\frac{481}{480}\]

    This is unfortunately greater than $1$. However, we can furthermore consider the moduli $d_1=16=2\cdot2^3$, $d_2=18=2\cdot3^2$, and $d_3=10=2\cdot5$. Notice that if any of them were equivalent modulo $2$, they would intersect exactly once every $\text{LCM}(d_i,d_j)$ integers, so we could subtract another $\frac{1}{\text{LCM}(d_i,d_j)}$ from our sum. Thankfully, we know at least two of them must share the same congruence class modulo $2$. In the worst-case scenario, we would have $16$ and $18$ together, since $\text{LCM}(16,18)=144$ is the highest of the three LCMs. Regardless, we still see $\frac{481}{480}-\frac{1}{144}=\frac{1433}{1440}<1$. Thus we do not have a covering.
\end{proof}

Note that this proof also shows we do not even have an infinite covering system with minimum modulus of 10 and the prime divisors of all moduli not exceeding $5$.

\section{Use of Integer Programming to show nonexistence of coverings} \label{IP}
In this section, we give further details on the use of integer programming in showing the nonexistence of certain covering systems, following the work of McNew and Setty \cite{McNew}. The first part of this section will follow along the lines of Section 3 in \cite{Klein}, while the second part of this section will use some new ideas discussed by McNew at a presentation at the University of South Carolina in November 2025. 

Suppose 
\[\mathcal{M}=\{d_1,\ldots,d_k\}\]
is a multiset of positive integers. The questions we wish to answer in the first part of this section are all of the form: Does there exist a covering system with a multiset of moduli $\mathcal{M}$, where $\mathcal{M}$ is some fixed set? The sets $\mathcal{M}$ under consideration will take the form $\mathcal{M}=\{d \geq m:d | L\}$. We now explain how such a problem may be viewed as an integer programming problem. 

Let $M=\{m_1,m_2,\ldots,m_n\}$ be the set of distinct integers in $\mathcal{M}$, let $L=\LCM(m_1,\ldots,m_n)$, and for all $i \in [n]$, let $f_i$ be the multiplicity of $m_i$ in $\mathcal{M}$, that is the number of times $m_i$ appears in $\mathcal{M}$. For every pair $(i,j) \in \Z^2$ such that $1 \leq i \leq n$ and $1 \leq j \leq m_i$, we define a binary integer variable $x_{i,j}$. One should think of $x_{i,j}$ as being $1$ if and only if the arithmetic progression $j \pmod{m_i}$ is in the covering system we are trying to construct. 

To ensure that every modulus $m_i$ is used at most $f_i$ times, we add for every $i \in [n]$ the constraint 
\[\sum_{j=1}^{m_i}x_{i,j} \leq f_i.\]
To ensure that every integer $b \in [1,L]$ is covered by some arithmetic progression in our covering system, for each such $b$ we add the constraint

\[\sum_{i=1}^n x_{i, \modd{b} {m_i}} \geq 1,\]
where $\modd{b} {m_i}$ is used to denote the least integer $c \in [m_i]$ such that $c \equiv b \pmod{m_i}$. 

The integer programming problem described above has a solution if and only if there is a covering system with a multiset of moduli $\mathcal{M}$. As explained in more detail in \cite{Klein}, in practice when setting up a covering problem as an integer programming problem, we usually preset a few arithmetic progressions. For every pair $(L,m)$ in the table below, we used integer programming to show that it was impossible to construct such covering systems. In the other columns, you can find the arithmetic progressions we preset to simplify the model, together with how long it took to determine that there were no solutions. 
\begin{center}
\begin{tabular}{|c|c|c|c|} \hline
    L & m &   Presets & Runtime (seconds)\\ \hline
    $2^2\cdot3^2\cdot5^2$ & 4 & N/A &   0.98      \\ 
    $2^3\cdot3^2\cdot5^2$ & 5&  N/A &    21.73      \\
    $2^3\cdot3^3\cdot 5^3$ & 6& $\{7 \pmod 8, 8 \pmod 9, 24 \pmod{25}\}$  &  9416.54    \\
    $2^4\cdot 3^4 \cdot 5^1$ & 6& $\{7 \pmod 8, 8 \pmod 9\}$   &  752.75  \\ 
    \hline
    \end{tabular}
\end{center}

The second part of this section will focus on using integer programming to find the maximal number of residues that it is possible to cover modulo $2^n \cdot 3^2 \cdot 5$, as $n$ ranges from $1$ to $10$. The methodology is inspired by a talk given by Nathan McNew at the University of South Carolina in November 2025. 

We start with $\mathcal{M}$, $M$, $L$, the $f_i$'s, and the $x_{i,j}$'s defined as above, together with the constraints 
\[\sum_{j=1}^{m_i}x_{i,j} \leq f_i\]
for each $i \in [n]$. Here is where we differ from the setup above. We add a binary integer variable $y_i$ for every $i \in [L]$, and we think of $y_i=1$ if and only if $i$ is covered by one of the arithmetic progressions we have in our potential covering. Our goal will be to maximize the number of residues covered, which in our new setting corresponds to maximizing
\[\sum_{i \in [L]} y_i.\]
We must add some additional constraints to ensure that if $y_i=1$ for some $i$, then there is indeed one of the arithmetic progressions that covers that residue. Towards this goal, we add for every $b \in [L]$ the constraint 
\[\sum_{i=1}^n x_{i, \modd{b} {m_i}} \geq y_b.\]

With this setup, an optimal solution to the integer programming problem corresponds to a set of arithmetic progressions that maximizes the density covered and the value of  
\[\sum_{i \in [L]} y_i\]
for the optimal solution corresponds to the maximum number of residues covered modulo $L$ using arithmetic progressions with moduli coming from $\mathcal{M}$. 

Below, one can find the results of our computations using this idea, together with the computation time involved for each case. We decided to take note of the number of residues left uncovered instead of noting the number of residues covered, or the density of the covered set, as this allows us to see a pattern emerge much more easily. 

\begin{center}
\begin{tabular}{|c|c|c|} \hline
    L &   Least residues left uncovered modulo $L$& Runtime (seconds)\\ \hline
    $2^{9}\cdot3^2\cdot5$  & 52 &   35738      \\ 
    $2^8\cdot3^2\cdot5$ &  52 &      11912   \\
    $2^7\cdot3^2\cdot 5$ & 52  &  991   \\
    $2^6\cdot 3^2 \cdot 5$ &  52   &  232  \\ 
     $2^5\cdot 3^2 \cdot 5$ &  52   &   30 \\ 
      $2^4\cdot 3^2 \cdot 5$ &  52   &  4  \\ 
       $2^3\cdot 3^2 \cdot 5$ &  52   &  1 \\ 
        $2^2\cdot 3^2 \cdot 5$ &  51   & $<1$   \\ 
         $2^1\cdot 3^2 \cdot 5$ &  40   &  $<1$ \\ 
          $2^0\cdot 3^2 \cdot 5$ &  36   & $<1$   \\ 
    
    \hline
    \end{tabular}
\end{center}
We ran a similar computation for $L=2^{10} \cdot 3^2 \cdot 5$, and aborted the computation after 48 hours with no clear sign of progress.

{\sc Department of Mathematics, Cedar Crest College, Allentown, Pennsylvania 18104, USA}

{\it Email address}, Joshua Harrington {\bf Joshua.Harrington@cedarcrest.edu}

\vspace{0.25 in}
{\sc Department of Mathematics, University of South Carolina, Columbia, South Carolina 29208, USA}

{\it Email address}, Jonah Klein {\bf jonah.klein@sc.edu}

\vspace{0.25 in}
{\sc Department of Mathematics, University of South Carolina, Columbia, South Carolina 29208, USA}

{\it Email address}, Joshua Lowrance {\bf LOWRANJ@sc.edu}

\vspace{0.25 in}
{\sc Department of Mathematics, University of South Carolina, Columbia, South Carolina 29208, USA}

{\it Email address}, Ognian Trifonov {\bf trifonov@math.sc.edu}

\end{document}